\documentclass[12pt, a4paper]{article}
\usepackage{latexsym}
\usepackage{amsmath,amssymb,amsfonts}

\newtheorem{prethm}{{\bf Theorem}}

\newenvironment{thm}{\begin{prethm}{\hspace{-0.5
               em}{\bf.}}}{\end{prethm}}

\newtheorem{prepro}[prethm]{Proposition}

\newtheorem{prelem}[prethm]{Lemma}
\newenvironment{lem}{\begin{prelem}{\hspace{-0.5
               em}{\bf.}}}{\end{prelem}}

\newtheorem{precor}[prethm]{Corollary}

\newtheorem{preremark}{{\bf Remark}}

\newenvironment{rem}{\begin{preremark}\em{\hspace{-0.5
              em}{\bf.}}}{\end{preremark}}

\newtheorem{preexample}{{\bf Example}}

\newtheorem{preproblem}{{\bf problem}}

\newtheorem{preproof}{{\bf Proof.}}

\newenvironment{proof}[1]{\begin{preproof}{\rm
               #1}\hfill{$\Box$}}{\end{preproof}}

\renewcommand{\thefootnote}

\begin{document}
\title{Finite unitary ring with minimal  non-nilpotent  group of units }
\author{Mohsen Amiri, Mostafa Amini\\
{\footnotesize{\em  Departamento de Matemática-ICE-UFAM, 69080-900, Manaus-AM, Brazil}}\\
{\footnotesize{\em Department of Mathematics, Payame Noor University, Tehran, Iran}}}
\footnotetext{E-mail Address: {\tt mohsen@ufam.edu.br;\,mamini1356@yahoo.com} }
\date{}
\maketitle

\begin{quote}
{\small \hfill{\rule{13.3cm}{.1mm}\hskip2cm} \textbf{Abstract.}
  Let $R$ be a finite unitary   ring such that $R=R_0[R^*]$,  where $R_0$ is the prime ring and $R^*$ is not a nilpotent group. We show  that  if   all proper subgroups of $R^*$ are nilpotent groups, then the cardinality  of $R$ is a power of  2. In addition, if   $(R/Jac(R))^*$ is not a $p$-group,   then  either $R\cong M_2(GF(2))$ or  $R\cong M_2(GF(2))\oplus A$, where $M_2(GF(2))$ is the ring of $2\times 2$ matrices over the finite field $GF(2)$ and $A$ is a direct sum of copies of the finite field $GF(2)$.   \\[.2cm]
{
\noindent{\small {\it \bf 2010 MSC}\,:16P10, 16U60, 20D15.  }}\\
\noindent{\small {\it \bf Keywords}\,: Finite ring,  Group of units, Nilpotent groups.}}\\
\vspace{-3mm}\hfill{\rule{13.3cm}{.1mm}\hskip2cm}
\end{quote}

\section{Introduction}
\ \ \ \
The relations between rings and their groups of units are an interesting research subject.
In \cite{22},  Groza has shown that  if $R$ is a finite ring and at most one simple
component of the semi-simple quotient ring $R$ by it's     Jacobson radical    $Jac(R)$  is the field of order 2, then the group of units $R^*$ is a
nilpotent group if and only if $R$ is a direct sum of two-sided ideals that are
homomorphic images of group algebras of the form $SP$, where $S$ is a particular
commutative finite ring, $P$ is a finite $p$-group, and $p$ is a prime number. More
recently,  Dolzan improved some results of Groza and described the structure of an
arbitrary finite ring with a nilpotent group of units, see  \cite{kesa}.

Let $X$ be a class of groups.  We say that a group $G$ is a minimal non-$X$-group, if $G\not\in X$, and all proper subgroups of $G$ belong to $X.$
Minimal non-$X$-groups have been studied for various classes of groups $X$. For example, minimal non-abelian groups were studied  by Miller and Moreno in \cite{Miller}, while Schmidt in \cite{Sc} studied minimal non-nilpotent groups and he characterized such finite groups.    The natural question is what we can say about  a  finite ring  whose   group of units   is a minimal non-$X$-group. In this paper, we study finite rings with minimal non-nilpotent groups of units and we prove that the order of such rings is $2^n$ for some positive integer $n$.
More precisely,  we  prove the following theorem:

\begin{thm}\label{main1}
Let  $R$  be  a   unitary ring of finite cardinality $2^am$, $m$  be  an odd number and $R^*$  a minimal non-nilpotent group. Then, $|R|=2^a$. Also, if $(R/Jac(R))^*$ is not a $p$-group, then  either $R\cong M_2(GF(2))$ or $R\cong M_2(GF(2))\oplus A$, where $A$ is a direct sum of copies of the  finite field of order two.
\end{thm}

In this paper,  from now on  $R$ denotes a ring with an identity $1\neq 0$   and for an arbitrary finite set $X$,   $|X|$ denotes the number of elements in $X$. We denote the group of units of $R$ by $R^*$,   the order of element $x$ in $R^*$  by $o(x)$, and  the group generated by $x$ by $\langle x \rangle$.   The ring of $n\times n$ matrices over a ring $R$ is denoted by $M_n(R)$, and    the set of elements that commute with every element of $R$ is denoted  by $Z(R)$. The centralizer of the subset $X$ of $R$ is the set of all elements of $R$ which commute with every element of $X$ is denoted by  $C_R(X)$.  Also, for any
pair $a, b \in R$,  $[a, b] = ab-ba$ is the Lie product of $a$ and $b$ and $R_0[S]$ denotes the subring
of $R$ which is generated by $S\subseteq R$  over $R_0$, where $R_0$ is the prime subring of $R$. The characteristic of $R$ is denoted by $Char R$ and $GF(p^m)$ is the finite field of order $p^m$, where $p$ is a prime number.


\section{Results}
We begin with  the  following useful  lemma:

\begin{lem}\label{d}
Let $R$ be a unitary finite  local ring with a nontrivial minimal ideal $I$ and we assume that $Jac(R)$ is a commutative ideal. Then  $Jac(R) \subseteq ann_R(I)$.
\end{lem}
\begin{proof}
{ We have  $Jac(R) I \subseteq I$, and since $I$ is minimal, we deduce that  $Jac(R) I = 0$ or $Jac(R) I = I$. Since $I\neq 0$, by Nakayama's Lemma, we conclude that  $Jac(R) I\neq I$. Therefore
 $Jac(R) I = 0$.}
\end{proof}
\begin{rem}\label{ddd1}
Let $R=A\oplus B$ be a  finite ring where $A$ and $B$ are two ideals in $R$. Then $R^*=A^*\oplus B^*$ and $1=1_A+1_B$ where $1_A$ and $1_B$ are the identity elements of $A$ and $B$, respectively.
It is clear that $A^*+1_B\leq R^*$ and that $A^*+1_B\cong A^*$.
\end{rem}

\begin{lem} \label{www} Let $R$ be a finite ring. If $|R|$  is an odd number, then $R = R_0[R^*]$.
\end{lem}
\begin{proof}
{By    Lemma 1.1 from \cite{22}, the proof is clear.}
\end{proof}

 Minimal non-nilpotent groups are characterized by Schmidt as follows:

\begin{thm}{\rm (see (9.1.1) of \cite{Rob})} \label{11} Assume that every maximal subgroup of a finite group $G$
is nilpotent but $G$ itself is not nilpotent. Then:

(i) $G$ is soluble.

(ii) $|G| = p^mq^n $ where $p$ and $q$ are unequal prime numbers.

(iii) There is a unique Sylow $p$-subgroup $P$ and a Sylow $q$-subgroup $Q$ is cyclic.
Hence $G = QP$ and $P\trianglelefteq G $.

\end{thm}
Let    $R$  be  a finite ring  with $|R|$  an odd number. Then, in the following theorem, we show that  $R^*$ is not a minimal non-abelian group.
\begin{thm}\label{hh}
Let $R$ be a finite ring of order $m$ with  $m$ an odd number. If every proper subgroup of  $R^*$ is an abelian group, then $R$ is a commutative ring.

\end{thm}
\begin{proof}{Consider the finite ring $R$
which is minimal subject to hypothesis of the theorem. Since every maximal subgroup of $R^*$ is  abelian,  $R^*$ is a minimal non-abelian group. By Lemma of \cite{1112}(p.512), every unitary  ring    of order $p$ or  $p^2$ for prime number $p$,  is a commutative ring. So we may assume  that  $|R|\not\in\{p,p^2\}$. Let   $S$ be a proper subring of $R$.   From Lemma \ref{www}, it follows  that $R=R_0[R^*]$, and hence $S^*\neq R^*$. By assumption,  $S^*$ is an abelian  group and by Lemma \ref{www}, $S=S_0[S^*]$  is a commutative ring.
 So every proper subring of $R$ is a commutative ring.
 Let $|R|=p_1^{\alpha_1}...p_k^{\alpha_k}$ be the canonical decomposition of $|R|$ to the prime numbers $p_i$. Then we know that $$R = R_1\bigoplus R_2\bigoplus ...\bigoplus R_k,$$ where each ideal $R_i$ is of order $p_i^{\alpha_i}$. Let $H_i$ be a subgroup of $R^*$ such that $H_i\cong R_i^*$ for all $i$.  If $k>1$, then $ H_i  $ is an  abelian subgroup of $R^*$. By minimality of $R$,  $R_i$ is a commutative ring for all $i$ and then $R$ is a commutative ring, which is a contradiction. So, suppose that $|R|=p^{\beta}$, where $p>2$ is a prime number.
 We have the following two cases with respect  to the  Jacobson radical,  either $Jac(R)=0$ or $Jac(R)\neq 0$:

 {\bf Case 1. } Let  $Jac(R)=0$. By  Wedderburn Structure Theorem,  $R\cong \bigoplus_{i=1}^t M_{n_i}(D_i)$ where $D_i$ is a finite field.  If $t>1$, then by minimality of $R$ and Remark \ref{ddd1}, every $M_{n_i}(D_i)$ is  commutative, and so $R$ is  commutative, which is a contradiction. It follows that $t=1$, and so $R\cong M_n(D)$,  where $D$ is a finite field and $n$ is a positive integer.
Since $R$ is  not a commutative ring, we have   $n>1$.  So, $(M_n(D))^*\cong R^*$ implies  that $R^*$ is  not a nilpotent group, hence $R^*$ is a minimal non-nilpotent group. But any Sylow $2$-subgroup of $R^*$ is neither  cyclic nor normal, which is a contradiction by Theorem \ref{11}(iii).

  {\bf Case 2.} Let  $Jac(R)\neq 0$.
 First suppose that $R^*$ is a nilpotent group.
 Since $-1\in R^*$ and $o(-1)=2$,  we deduce that $2\mid |R^*|$.  Since $Jac(R)\neq 0 $, by Lemma 1.2 of \cite{22}, $1+Jac(R)$ is a $p$-group. Let $P\in Syl_p(R^*)$ and $K$ be a subgroup of  $R^*$ such that $R^*=PK$ and $P\cap K=1$. By assumption, $P$ and $K$ are abelian groups, and  hence $R^*$ is an abelian group. Thus by Lemma \ref{www}, $R=R_0[R^*]$ is commutative, which is a contradiction. It follows that  $R^*$ is not a nilpotent group, and it follows that $R^*$ is a minimal non-nilpotent group.
 By Theorem \ref{11}, $|R^*|=r^m q^n$  where $r,q$ are  prime numbers.
 Also, since $2\mid |R^*|$ and  by Lemma (1.2) of \cite{22}, $1+Jac(R)$ is a $p-$group, we may assume that $r=p$ and $q=2$.
  On the other hand, $({R}/{Jac(R)})^*={R^*+Jac(R)}/{Jac(R)}$, and so every proper subgroup of $({R}/{Jac(R)})^*$ is an abelian group.  By minimality of $R$,  ${R}/{Jac(R)}$ is a commutative ring and hence, we deduce that $[R,R]\subseteq Jac(R)$.  Let $P\in Syl_p(R^*)$ and $Q\in Syl_2(R^*)$. By Theorem \ref{11}(iii), either $P\lhd R^*$ or $Q\lhd R^*$. We claim that $P\lhd R^*$. Otherwise,  by Theorem \ref{11}(iii), $P=\langle z\rangle$ is a cyclic  subgroup of $R^*$, where $z\in P$.  Since $1+Jac(R)\leq P$, there is a positive  integer $i$ such  that $H=1+Jac(R)=\langle z^i\rangle$. Since $H\lhd R^*$ and $R^*$ is a non-nilpotent group,  $H\neq P$. Since
  $HQ$ is an abelian subgroup of $R^*$ and $R=\langle z, Q\rangle$, we have $H\leq Z(R^*)$, and  so $H\leq Z(R)$. Consequently,
  $Jac(R)\subseteq Z(R)$. Since  $[R,R]\subseteq Jac(R)$, we have $uv-vu\in Jac(R)$ for all $u,v\in R^*$. Thus, it follows that  $uvu^{-1}v^{-1}-1\in Jac(R)\subseteq Z(R)$,  and so  $uvu^{-1}v^{-1}\in Z(R)$ for all $u,v\in R^*$. Hence $R^*$ is a nilpotent group, which is a contradiction. Therefore $P\lhd R^*$, as claimed. By Theorem \ref{11}(iii), $Q=\langle x \rangle$ for some $x\in Q$. Now, we claim that $R$ is a local ring.
 Let $\{M_1,...,M_k\}$ be the set of all maximal ideals of $R$, where $k>1$. Since ${R}/{Jac(R)}={R}/{(M_1\cap ...\cap M_k)}\cong {R}/{M_1}\times ....\times {R}/{M_k}$, we conclude that $({R}/{Jac(R)})^*\cong ({R}/{M_1})^*\times ....\times ({R}/{M_k})^*$.
  Let $\overline{Q}\in Syl_2( ({R}/{M_1}\times ....\times {R}/{M_k})^*)$.
  Then, since $2\mid |({R}/{M_1})^*|$ for all $1\leq i\leq k$, we see that $\overline{Q}$ is not a cyclic group. Also, we have $({R}/{Jac(R)})^*={R^*+Jac(R)}/{Jac(R)}={PQ+Jac(R)}/{Jac(R)}$, and so $Q$ is not a cyclic group,
 which is a contradiction. Hence $k=1$, as claimed. Let $M=Jac(R)$. Since $1+M$ is an abelian group,   $M$ is a commutative ideal. Since $R/M$ is a finite field, $M$ is not a central ideal. So, there exists  $w\in M$ such that $wx\neq xw$. By minimality of $R$, we have $R=R_0[ w,x ]$. Let $I$ be a minimal ideal of $R$.
 We follow the proof by separating   two  subcases, either $Z(R)\cap I\neq 0$ or  $Z(R)\cap I=0$:

{\bf Subcase 1.}  Let $0\neq a\in Z(R)\cap I$. By Lemma  \ref{d}, $M\subseteq ann_R(I)=\{r\in R:\ rs=0 \ for \ all \ s\in I \}$. Since $R/M$ is a finite field,  $M= ann_R(I)$. It follows from  $a\in Z(R)$ that  $I=Ra$ is two sided ideal. Since $R/M$ is a finite field,   $(R/M)^*=\langle x+M\rangle$ for some $x\in R$ with $gcd(o(x+M),p)=1$. Let $y\in R\setminus M$. Therefore  $y+M=x^i+M$ for some integer $0\leq i\leq n-1$. Then  $y=x^i+s$ for some $s\in M$ and this, implies that $ya=x^ia+sa=x^ia$. So, we have $I=\{0,xa,...,x^na\}\subseteq M$.
Since  $xx^ia=x^iax$,  $w(x^ia)=(x^ia)w$ and $R=R_0[x,w]$, we conclude  that $x^ia\in Z(R)$, and so $I\subseteq Z(R)$.  The minimality of $R$ implies that ${R}/{I}$ is a commutative ring, and consequently $[R,R]\subseteq I\subseteq Z(R)$. Thus, we have
   $uvu^{-1}v^{-1}-1\in I\subseteq Z(R)$ for all $u,v\in R^*$. It follows that $uvu^{-1}v^{-1}\in  Z(R^*)$, and hence $R^*$ is a nilpotent  group, which is a contradiction.

{\bf Subcase 2.} Let $ Z(R)\cap I=0$ and $0\neq b\in I$. Since  $R=R_0[w,x]$ and $M$ is commutative, we have $bw=wb$, and so    $[b,x]\neq 0$. Therefore $R=R_0[ b,x ]$. We may assume that $b=w\in I$ and $m_1,m_2\in M$. Since  $M$ is a commutative ring and $xm_1,\, m_2x\in M$, we have
 $$(xm_1)m_2=m_2(xm_1)=(m_2x)m_1=m_1m_2x.$$ Since $R=R_0[ x,w]$ and $wm_1m_2=m_1m_2w$, we conclude that   $M^2\subseteq Z(R)$.  If $M^2\neq 0$, then by minimality of $R$,  ${R}/{M^2}$ is a commutative ring, and so $0\neq [R,R]\subseteq M^2\cap I$. Since $I$ is a minimal ideal and $M^2$ is an ideal,
 $I\subseteq M^2\subseteq Z(R)$, which is a contradiction. Hence  $M^2=0$, and so by considering $R$ as a local ring,  for all $s\in M\setminus\{0\}$,  we have  $M\subseteq ann_R(s)$. We claim that $I=M$, otherwise consider $l\in M\setminus I$. Since $R=R_0[x,w]$, we have  $l=\sum n_ix^i+c$, where
 $c\in I$ and $n_i\in R_0$. Therefore $l-c=\sum n_ix^i\in M$.  Then $\sum n_ix^i\in Z(R)$.  If $l-c\neq 0$, then by minimality of $R$, we have that $R/R(l-c)$ is a commutative ring. By a similar argument to the one subcase 1, we have $R(l-c)\subseteq Z(R)$. From $0\neq [R,R]\subseteq R(l-c)\cap I$, where $I$ is a minimal ideal,   follows  that $I\subseteq R(l-c)\subseteq Z(R)$, which is a contradiction. Then
 $l-c=0$, which is a contradiction. Therefore $M=I$.  Since $R/M$ is a finite field, we have $R=R^*\cup M$.
  Then $|R^*|=|R|-|M|=o(x)|P|$. Let $|R/M|=p^{m}$. Then $|R|=p^m|M|$, and so $|R|-|M|=(p^m-1)|M|$, consequently,  $1+M=P$.
Since $P\langle x^2 \rangle$ is an abelian subgroup of $R^*$ and $R=R_0[ w, x]$, we conclude that $x^2\in Z(R)$. If
 $p\mid o(x+h)$ for some $h\in P$, then   $e=(x+h)^{\frac{o(x+h)}{p}}\in P= 1+M$, and so $e-1\in  M$.  Since $R=R_0[ x+h, w]$ and $M$ is commutative, $e-1\in Z(R)$. Therefore $M\cap Z(R)\neq 0$, which is a contradiction. It follows that $gcd(o(x+h),p)=1$ for all $h\in M$.
By a similar argument to the one  above,  $(x+h)^2\in Z(R)$ for all $h\in M$. Then $x^2+xh+hx+h^2\in Z(R)$. Since $h^2=0$ and $x^2\in Z(R)$, we have $xh+hx\in Z(R)\cap M=0$.  Therefore
   $xh=-hx$ for all $h\in M$. Let $0\neq h\in M$.    Since $({R}/{M})^*=\langle x+M\rangle$, we deduce that
   $x+1+M=x^t+M$ for some integer $t$, and then $q=x^t-x-1\in M$. Since $M$ is a commutative ideal, $qw=wq$. It follows from $qx=xq$ and from $R=R_0[ x,w]$ that
 $q=x^t-x-1\in M\cap Z(R)=0$. So,
   $x+1=x^t$. Since $x+1=x^t\not\in Z(R)$, we deduce that $t$ is an odd number. Then, we have  $x^th=(-1)^thx^t=-hx^t$.  Therefore,
   $(x+1)h=-h(x+1)$. But $(x+1)h=xh+h=-hx+h=-h(x+1)$. Hence  $2h=0$, and so $h=0$ for all $h\in M$, which is a  contradiction.}

\end{proof}

\begin{thm}\label{thm222}
Let $R$ be a finite ring of order $p^m$ where $p$ is an odd prime number. If every proper subgroup of  $R^*$ is  nilpotent, then $R$ is a commutative ring.

\end{thm}
\begin{proof}{Consider the finite ring $R$
which is minimal with respect to  these conditions, but  it  is not commutative.  Then  $R^*$ is a minimal non-nilpotent group.   By a similar argument to the one in case 1 of the previous theorem, we may assume that $Jac(R)\neq 0$.  By Theorem \ref{11}, $|R^*|=r^m q^n$  where $r$ and $q$ are  prime numbers. By Lemma 1.2 from \cite{22}, $1+Jac(R)$ is a $p$-group and then
  $r=p$.  Since $-1\in R^*$ and $o(-1)=2$,  we have $q=2$.  Let $P\in Syl_p(R^*)$ and $Q\in Syl_2(R^*)$.
 Let $I$ be a minimal ideal of $R$ that is contained in $Jac(R)$. Then $I^2=0$, and hence $I$ is  commutative.
By minimality of $R$, we have that $R/I$ is  commutative, so $[R,R]\subseteq I$. We have two cases with respect  to  $I\cap Z(R)$,  either $I\cap Z(R)\neq 0$ or $I\cap Z(R)=0$

{\bf Case 1. } Suppose that  $I\cap Z(R)\neq 0$. If $0\neq c\in I\cap Z(R)$, then by a similar argument to the one in the  subcase 1 of case 2, in the  above theorem,  $I\subseteq Z(R)$, and  then $1+I\leq Z(R^*)$. Since $uvu^{-1}v^{-1}-1\in I$, we have $(R^*)'\leq Z(R^*)$, and so $R^*$ is a nilpotent group, which is a contradiction.

{\bf Case 2. } Let $I\cap Z(R)=0$.
  First suppose that $1+I\neq P$.  Since $1+I\lhd P$, there exists
 $c\in I\setminus\{0\}$ such that $1+c\in Z(P)$. Also, $(1+I)Q$ is a proper nilpotent subgroup of $R^*$. So $1+I\leq C_{R^*}(Q)$, and hence $c\in Z(R)$, which is a contradiction. Therefore $1+I=P$ is an abelian subgroup of $R^*$.  By Theorem \ref{11}, either $P$ is cyclic  or $Q$ is cyclic. Since $1+I=P\lhd R^*$, by Theorem \ref{11}, $Q$ is  cyclic.
 Since $ P$ is an abelian group and $Q$ is   cyclic, every proper subgroup of $R^*$ is an abelian group, which is a contradiction by Theorem \ref{hh}.
}

\end{proof}

If $2\mid |R|$, then the above theorem is no longer valid. For example, let $R$ be the set of all $2\times 2$  matrices   over the finite field $GF(2)$. Then
$R^*\cong S_3$, where $S_3$ is the symmetric group of order $6$ and clearly, $S_3$ is a minimal
non-abelian group. For simplicity,  let   $\Delta$ be  the set of all rings $R$  in which  either    $ R\cong M_2(GF(2))$ or $ R\cong M_2(GF(2))\oplus A $,       where $A$ is   a direct sum of copies of the finite field of order two.
 \begin{rem}\label{ttt1} Let $Sl(2,GF(2^m))$ be the kernel of the homomorphism $det (M_n(GF(2^m))\longrightarrow GF(2^m)^*.$
 We recall  that when $m>1$, then  $Sl(2,GF(2^m))^*=((Sl(2,GF(2^m)))^*)'$, and hence for $n>1$ and $m>1$, we have that $(M_n(D))^*$ is not a minimal non-nilpotent group.
\end{rem}
\begin{thm}\label{thm2225}
Let  $R$  be  a  unitary    non-commutative ring of finite cardinality $2^{\beta}$ such that every proper subgroup of  $R^*$ is a nilpotent group.

(a) If $Jac(R)=0$, then $R\in \Delta$.

(b) If $Jac(R)\neq 0$, then
 $(R/Jac(R))^*$ is  a cyclic $p$-group for some odd prime number $p$.
\end{thm}
\begin{proof}{
(a) We proceed   by induction on $\beta$.  Since $R$ is a simple artinian ring,   by the structure theorem of Artin-Wedderburn, we have   $R\cong \bigoplus_{i=1}^t M_{n_i}(D_i),$ where  every $D_i$ is a finite field.  If $t=1$, then $R\cong M_{n_1}(D)$, and clearly, $n_1=2$ and $D\cong GF(2)$, so $R\in \Delta$.  Let $t>1$ and $n_i>1$ for some $1\leq i\leq t$. By Remark \ref{ttt1}, if $M_{n_i}(D_i)$  is a minimal non-abelian group,  then, $n_i=2$ and $D_i=GF(2)$.  If for some $j\neq i$, we have  $n_j>1$, then  $R^*$ is not  a minimal non-abelian group, which is a contradiction. Therefore $M_{n_j}(D_j)\cong D_j$ for all $j\neq i$. Let $H\leq R^*$ such that $H\cong (M_{n_i}(GF(2)))^*$. If
 $|D_j^*|>1$ for some $j\neq i$, then  $R^*\neq H$, and hence $H$ is a non-nilpotent proper subgroup of $R^*$, which is a contradiction.
Consequently, $D_j\cong GF(2)$, and hence $R\in \Delta$.

(b) Suppose for a contradiction that $(R/Jac(R))^*$ is not  a $p$-group. We may assume that $|R|$ is minimal  such that  $(R/Jac(R))^*$ is not a $p$-group.
By Theorem \ref{11}, $R^*=PQ$, where $P\lhd G$ and $Q$  is a cyclic
Sylow subgroup.
  Let $I \subseteq Jac(R)$ be a minimal ideal of $R$. It is easy to see that  $char(I)=2$ and $I^2=0$.  Therefore, $1+I$ is an elementary abelian $2$-group.
  Also, $(R/I)/(Jac(R/I))\cong R/Jac(R)$ implies  that $((R/I)/(Jac(R/I)))^*$ is not a $p$-group, so by minimality of $R$,
  $(R/I)^*$ is  a nilpotent group.
   Let $p>2$ be a prime number such that
  $p\mid |R^*|$. Clearly,  $2p\mid |(R/Jac(R))^*$.   Let $\{M_1,...,M_k\}$ be the set of all maximal ideals of $R$. Then, we have
   $$R/Jac(R)\cong R/M_1\oplus ...\oplus R/M_k.$$
   Since $R/M_i$ is a simple ring,  $R/M_i\cong M_{n_i}(GF(2^{m_i}))$ for  some positive integers $n_i$ and $m_i$.
   If $n_i>1$ for some i, then $R/M_i\cong M_{n_i}(GF(2^{m_i}))$. Let   $x$ and $y$ be two arbitrary elements of $ R^*$ such that $xy\neq yx$ and $gcd(o(x),o(y))=1$.
Since  $(R/I)^*$ is a nilpotent group, we have
  $xy-yx\in I\subseteq Jac(R)$. Hence $n_i=1$ for all $i$, and so $R/M_i$ is a finite field.
  But $gcd(|(R/M_i)^*|,2)=gcd(2^{m_i}-1,2)=1$ for all $i$, so
  $2\nmid |(R/Jac(R))^*$, which is a contradiction.

}
\end{proof}

Here we give an example for the  statement  of  Theorem \ref{thm2225} part (b). Let $GF(2)[x,y]$ be the free ring generated with two elements $x$ and $y$ over finite field $GF(2)$.    Let $H$ be the ideal generated by
  $\{x^2, y^3+y+1, xy-y^2x\}$. Let $R=GF(2)[x,y]/H$, and let $I$ be the ideal generated with
  $x+H$ in $R$.
  Since $xy-yx\not\in H$,  $R$ is a non-commutative ring. Let $L$ be the ideal generated by $t^3+t+1$ in $Z_2[t]$. Since $R/I\cong Z_2[t]/L$ and
  $L$ is a maximal ideal,  $R/I$ is a finite field of order $8$. Let $(R/I)^*=\langle u+I\rangle$.
  It is easy to check that $I=\{0,x+H,ux+H,...,u^7x+H\}$.  Clearly, $1+I$ is an elementary abelian $2$-group and
  $R^*=(1+I)\langle u\rangle$ is a minimal non-nilpotent group.

 Now we are ready to prove   Theorem \ref{main1}.
\begin{proof}{   By Theorem \ref{thm222},
  $\alpha_1\geq 1$.
Let $|R|=2^{\alpha_1}p_2^{\alpha_2}...p_k^{\alpha_k}$ be the canonical decomposition of $|R|$ to the prime numbers $p_i$. Then  \[R = R_1\oplus R_2\oplus\cdots\oplus R_k,\] where each ideal $R_i$ is of order $p_i^{\alpha_i}$.  By Theorem  \ref{thm222},   $R_2\oplus\cdots\oplus R_k$ is a commutative ring, and hence  $(R_1)^*$ is a minimal non-nilpotent group. Consequently, $k=1$.   The rest of proof is clear by Theorem \ref{thm2225}.
}
\end{proof}

\end{document}